\documentclass{article}
\usepackage{graphicx} % Required for inserting images
\usepackage{url}
\usepackage{amsmath,amsthm,amsfonts,amssymb}
\usepackage[colorlinks=true,citecolor=black,linkcolor=black,urlcolor=blue]{hyperref}
\usepackage{cleveref}
\usepackage{float}
\usepackage{enumitem}
\usepackage{todonotes}
\usepackage{dsfont}
\usepackage{amstext}
\usepackage{algorithm}
\usepackage{algpseudocode}

\newtheorem{theorem}{Theorem}

\newtheorem{lemma}[theorem]{Lemma}

\newtheorem{question}[theorem]{Question}

\DeclareMathOperator{\rad}{rad}

\def\paren#1{\left( #1 \right)}
\def\acc#1{\left\{ #1 \right\}}

\usepackage{xcolor}

\renewcommand{\le}{\leqslant}
\renewcommand{\ge}{\geqslant}

\title{On odd perfect numbers with exactly one even exponent greater than 2}
\author{Pascal Ochem and Joshua Zelinsky}
\date{2026}

\begin{document}

\maketitle
\begin{abstract}
We show that if all the even exponents of an odd perfect number $N$ are equal to 2 except for one, then $3^{23,000,000,000}$ divides $N$.
% An odd perfect number $N$ is of the form $N=p^e\cdot m^2$ where $p$ is the \emph{special prime}.
% A component of $N$ is \emph{big} if it is non-special and its (even) exponent is at least 4.
% It is known that $m$ cannot be square-free, i.e., $N$ has at least one big component.
% We say that $N$ is \emph{restrained} if it has exactly one big component.
% We show that if $N$ is restrained, then the big exponent is at least 100.
\end{abstract}

\section{Introduction}
In this paper, we consider that $N$ is an odd perfect number, that is,
$N$ is odd and $\sigma(N)=2N$, where $\sigma$ is the positive-divisor-sum function.
Euler proved that $N=p^e\cdot m^2$ where $p$ is prime,
$p\equiv e\equiv 1\pmod{4}$ and $p$ does not divide $m$.
We say that $p$, $e$, and $p^e$ are respectively the special prime,
the special exponent, and the special component.
McDaniel~\cite{McDaniel:1970} has shown that $m$ cannot be square-free,
i.e., at least one of the even exponents of $N$ is at least 4.
We say that $N$ is \emph{restrained} if $N$ has exactly one such exponent.
If $N$ is restrained, let $n^q$ denote the unique component such that $n\ne p$ and $q\ge4$.
We say that $n$, $q$, and $n^q$ are respectively the notable prime,
the notable exponent, and the notable component.
Finally, we say that a prime or a component of $N$ is \emph{ordinary} if its exponent is 2.
Thus, the only non-ordinary primes are $p$ and $n$.

Our aim is to get a lower bound on $q$.

\begin{theorem}\label{main}
If $N$ is a restrained odd perfect number, then $n=3$ and $q\ge23,000,000,000$.
\end{theorem}

This improves the known bound $q\ge16$ obtained by combining earlier results:
\begin{itemize}
    \item $q\ne2\pmod{3}$~\cite{McDaniel:1970}
    \item $q\ne4$~\cite{Cohen:1987}
    \item $q\ne6$~\cite{Kanold:1950}
    \item $q\not\in\acc{10,12}$~\cite{CohenWilliams:1985}
\end{itemize}

The programs and data used in this paper are available at\\ \url{https://www.lirmm.fr/~ochem/opn/restrained/}

We will use Pace Nielsen's lower bound on the number $\omega(N)$ of distinct prime
factors of an OPN.

\begin{theorem}\cite{Nielsen:2015}\label{Pace}
If $N$ is an odd perfect number, then $\omega(N)\ge10$.
\end{theorem}

We will also need the following result that follows as a special case of Proposition~1 from Bennett and Levin~\cite{{BennettLevin:2015}}.
\begin{lemma}\label{rtt}
There is no solution to $z^2+z+1=r^t$ where $z$ and $r$ are prime and $t>1$.
\end{lemma}

\begin{lemma}\label{NotFromNotable}
If $s$ and $n$ are odd primes such that $\operatorname{ord}_s(n)$ is even, then for every even $q$, $s\nmid\sigma(n^q)$.
\end{lemma}

\begin{proof}
Assume for the sake of contradiction that $s\mid\sigma(n^q)=\frac{n^{q+1}-1}{n-1}$. 
So $s\mid n^{q+1}-1$, which means that $n^{q+1}\equiv1\pmod s$. 
By definition, $\operatorname{ord}_s(n)$ divides the exponent $q+1$.
This is a contradiction since $\operatorname{ord}_s(n)$ is even and $q+1$ is odd.
\end{proof}

\begin{lemma}\label{Robin}
Let $x$, $s$, and $t$ be distinct primes such that $x\nmid t-1$ and $s\nmid t-1$. If $\operatorname{ord}_s(t)\mid\operatorname{ord}_x(t)$, then for every integer $v\ge 1$, $x\mid\sigma(t^v)$ implies $s\mid\sigma(t^v)$.
\end{lemma}

\begin{proof}
Assume that $x\mid\sigma(t^v)$. The divisor sum is given by $\sigma(t^v)=\frac{t^{v+1}-1}{t-1}$. 
Since $x\mid\sigma(t^v)$ and $x\nmid t-1$, it must be that $x\mid t^{v+1}-1$, meaning $t^{v+1}\equiv1\pmod x$. 

By the definition of multiplicative order, $\operatorname{ord}_x(t)$ must divide the exponent $v+1$. 
By our hypothesis, $\operatorname{ord}_s(t)\mid\operatorname{ord}_x(t)$. 
By transitivity of divisibility, this implies that $\operatorname{ord}_s(t)\mid v+1$. 

Because $\operatorname{ord}_s(t) \mid v+1$, it follows that $t^{v+1}\equiv1\pmod s$, meaning $s\mid t^{v+1}-1$.
Since $s$ is prime and $s \nmid t-1$, $s$ cannot divide the denominator of $\frac{t^{v+1}-1}{t-1}$. 
Therefore, $s$ must divide the quotient. Thus, $s\mid\sigma(t^v)$.
\end{proof}

\begin{lemma}\label{p1mod3}
If $N$ is a restrained odd perfect number, then $p\equiv1\pmod3$.
\end{lemma}

\begin{proof}
Suppose for the sake of contradiction that $p\equiv2\pmod3$.
This implies that $p$ is not produced by an ordinary component.
Since $p$ is obviously not produced by the special component, we must have $p\mid\sigma(n^q)$.
$p\equiv2\pmod3$ also implies that $3\mid p+1$, so that $3\mid N$.

If 3 is ordinary, i.e. $3^2\parallel N$, then the non-special components must produce at most one copy of $3$. 
So at most one ordinary prime $x$ is congruent to $1\pmod3$.
Thus $N$ contains at most two primes congruent to $1\pmod3$, namely $x$ and $n$.
Since $\omega(N)\ge10$ by Theorem~\ref{Pace}, excluding $p$, $n$, and $x$ implies that at least 7 ordinary primes $s$ are such that $\rad\paren{\sigma\paren{s^2}}\mid x\cdot n$.
Now $N$ contains exactly two copies of the ordinary prime $x$, so that at least 5
ordinary primes $s$ are such that $\rad\paren{\sigma\paren{s^2}}\mid n$.
That is, we must have 5 distinct solutions to $s^2+s+1=n^a$, which contradicts Lemma~\ref{rtt}.

If 3 is notable, we have $p\mid\sigma(3^q)$.
This implies $3^{q+1}\equiv1\pmod p$.
Therefore $\operatorname{ord}_p(3)$ must divide the odd exponent $q+1$.
So $\operatorname{ord}_p(3)$ is odd, which implies that 3 is a quadratic residue modulo $p$.
Since $p\equiv1\pmod4$, we have $p\equiv 5\pmod{12}$ and we write $p=12k+5$.
Now we compute the suitable Legendre symbol.
\begin{align*}
\paren{\tfrac3p} & = \paren{\tfrac p3}\cdot\paren{-1}^{\tfrac{3-1}2\cdot\tfrac{p-1}2}\text{ by quadratic reciprocity}\\
 & = \paren{\tfrac p3}\cdot\paren{-1}^{6k+2}\\
 & = \paren{\tfrac p3}\\
 & = \paren{\tfrac{12k+5}3}\\
 & = \paren{\tfrac23}\\
 & = -1
\end{align*}
This means that 3 is a quadratic non-residue modulo $p$, a contradiction.

We have a contradiction in both cases, therefore $p\equiv1\pmod3$.
\end{proof}

We say that an ordinary prime factor $x$ of $N$ is \emph{regular} if $\rad\paren{\sigma\paren{x^2}}\mid p\cdot n$. Every prime factor of $N$ that is not regular is \emph{irregular}. We also extend the definition of regular and irregular to a component according to its prime.

Throughout the paper, we use factor trees to rule out a given prime factor of $N$.
A branch of a tree is cut when we reach one the following contradictions.
We can use one of the previous lemmas.
If the abundancy is already greater than 2, we write $\sigma_{-1}(\cdotp)>2$ in the paper and $a(\cdotp)>2$ in the data files.
If a prime $t$ is assumed to be ordinary and at least three copies of $t$ are already produced by the known components, we write XS $t$.
If a prime $t$ is produced and $t$ has been previously ruled out as a factor of $N$, we write F $t$.

Since $N$ is restrained, the exponent of 3 in $N$ is in $\acc{0,2,q}$.
The proof of Theorem~\ref{main} is split according to these three cases in the following sections.

\section{When 3 is not a factor}
In this section, we suppose that $N$ is a restrained odd perfect number that is not divisible by 3 and we show that no such $N$ exists. We begin with structural properties of $N$.

\begin{lemma}\label{struct}
Let $N$ be as stated and let $s$ be an ordinary prime. Then
\begin{enumerate}
\item $s\equiv5\pmod{6}$,
\item $s$ is regular, i.e. $\sigma(s^2)=p^a\cdot n^b$,
\item $n\equiv1\pmod{6}$.
\end{enumerate}
\end{lemma}

\begin{proof}
First, $s\equiv2\pmod{3}$ since otherwise $3\mid\sigma(s^2)$.
Now, if a prime $\alpha$ divides $\sigma(s^2)$, then $\alpha\equiv1\pmod3$.
Thus $\alpha$ is not ordinary, that is, either $\alpha=p$ or $\alpha=n$.
So every ordinary prime $s$ is regular.
By Lemma~\ref{rtt}, either $\sigma(s^2)=p$ or $n\mid\sigma(s^2)$. Since $N$ has at least 2 distinct ordinary primes, there exists $s$ such that $n\mid\sigma(s^2)$. Therefore $n$ is congruent to $1\pmod3$.
\end{proof}

Now, we upper bound the abundancy of $N$ to show that it cannot reach 2.
We break it into 4 cases:
\subsection{Special and notable components}
Since $p\equiv1\pmod{12}$ and $n\equiv1\pmod{6}$, we obtain
$\sigma_{-1}\paren{p^e n^q}\le\sigma_{-1}\paren{13^\infty\cdot 7^\infty}=\tfrac{13}{12}\cdot\tfrac76=\tfrac{91}{72}$.
So we set $$A_1=\tfrac{91}{72}.$$

\subsection{Sporadic ordinary components}
Now we bound the abundancy of the ordinary components as follows.
By Lemma~\ref{struct}, every ordinary prime is regular.
Every pair $(a,b)\in\mathds{N}^2$ determines at most one ordinary
prime $s$ such that $\sigma(s^2)=p^a\cdot n^b$.
We do not consider fixed values of $p$ and $n$.
%We only remember that $p\equiv1\pmod4$.
Let $m(a,b)$ be the smallest solution to $\sigma(s^2)=p^a\cdot n^b$
where $s$, $p$, and $n$ are prime, $p\equiv1\pmod4$, and $s\equiv2\pmod3$ (since otherwise $3\mid\sigma(s^2)$).
The abundancy of the ordinary components is then at most
$$\prod_{(a,b)\in\mathds{N}^2}\tfrac{m(a,b)}{m(a,b)-1}.$$
Using the program pairs.py, we have checked every prime $s\equiv5\pmod{6}$ up to $10^9$ and found a few values of $m(a,b)$.
\begin{table}[!htb]
\centerline{
\begin{tabular}{|r|r|r|r|l|}
\hline
$(a,b)$ & $s=m(a,b)$ & $p$ & $n$ & comment \\
\hline
\hline
$(0,1)$ & 5 & - & 31 &  \\
\hline
$(1,1)$ & 29 & 13 & 67 &  \\
\hline
$(1,0)$ & 59 & 3541 & - &  \\
\hline
$(1,2)$ & 863 & 15217 & 7 &  \\
\hline
$(2,1)$ & 1667 & 13 & 16453 &  \\
\hline
$(1,3)$ & 1733 & 8761 & 7 &  \\
\hline
$(1,4)$ & 2819 & 61 & 19 &  \\
\hline
$(3,1)$ & 306419 & 13 & 42736873 &  \\
\hline
$(1,5)$ & 1345913 & 107781469 & 7 &  \\
\hline
$(1,6)$ & 2858543 & 69454657 & 7 & discarded since $a+b\ge7$ \\
\hline
$(4,1)$ & 3619619 & 13 & 458725021 &  \\
\hline
$(1,7)$ & 30388409 & 1121320237 & 7 & discarded since $a+b\ge7$ \\
\hline
$(5,1)$ & 68839637 & 13 & 12763223899 &  \\
\hline
$(1,8)$ & 453031331 & 35601816493 & 7 & discarded since $a+b\ge7$ \\
\hline
$(6,1)$ & 993359207 & 13 & 204433719073 & discarded since $a+b\ge7$ \\
\hline
\end{tabular}}
\caption{Smallest solutions to $\sigma(s^2)=p^a\cdot n^b$.}\label{tab1}
\end{table}

The abundancy due to the pairs in Table 1 such that $a+b\le6$ is at most
$$A_2=\tfrac54\cdot\tfrac{29}{28}\cdot\tfrac{59}{58}\cdot\tfrac{863}{862}\cdot\tfrac{1667}{1666}\cdot\tfrac{1733}{1732}\cdot\tfrac{2819}{2818}\cdot\tfrac{306419}{306418}\cdot\tfrac{1345913}{1345912}\cdot\tfrac{3619619}{3619618}\cdot\tfrac{68839637}{68839636}.$$
% 5/4*29/28*59/58*863/862*1667/1666*1733/1732*2819/2818*306419/306418*1345913/1345912*3619619/3619618*68839637/68839636
%         1.32051931736840940811

\subsection{Other pairs $(a,b)$ with $a+b\le6$}
Some pairs $(a,b)$ with $a+b\le6$ cannot correspond to a solution to $\sigma(s^2)=p^a\cdot n^b$.
\begin{itemize}
\item The pair $(0,0)$.
\item The pairs $(a,0)$ with $a\ge2$ and $(0,b)$ with $b\ge2$, by Lemma~\ref{rtt}.
\item The pairs $(2,2)$, $(2,4)$, and $(4,2)$: since $s^2<\sigma(s^2)<(s+1)^2$), then $\sigma(s^2)$ is not a square.
\end{itemize}

There remain 14 pairs: the 11 pairs in Table 1 and the 3 pairs $(2,3)$, $(3,2)$, and $(3,3)$.
The previous computation implies that $m(a,b)>10^9$ for these 3 pairs. 
%There are most $\tfrac{7\cdot8}2=28$ pairs with $a+b\le6$.
Therefore we can upper bound their contribution to the abundancy by $$A_3=\paren{1+10^{-9}}^3.$$

\subsection{Pairs $(a,b)$ with $a+b\ge 7$}
From $\sigma(s^2)=p^a\cdot n^b$, we set $k=a+b$ and we deduce $(s+1)^2\ge7^k$, since $p\ge7$ and $n\ge7$.
So $s\ge7^{\tfrac k2}-1$. Therefore the abundancy due to $(a,b)$ is at most $1+\tfrac1{s-1}\le1+\tfrac1{7^{k/2}-2}$.
Lemma~\ref{rtt} rules out the pairs $(k,0)$ and $(0,k)$.
So there are $k-1$ pairs such that $a+b=k$.
Then a bound on the abundancy is the infinite product
$$A_{\infty}=\prod_{k\ge7}\paren{1+\tfrac1{7^{k/2}-2}}^{k-1}.$$

\subsection{The contradiction}
We have
\begin{itemize}
\item $A_1\approx1.26388888888888888888$
\item $A_2\approx1.32051931736840940811$
\item $A_3\approx1.00000000300000000300$
\item $A_{\infty}\approx1.01178747704420359446$
\end{itemize}    
Finally, we get $\sigma_{-1}(N)\le A_1\cdot A_2\cdot A_3\cdot A_{\infty}\approx1.68866287554178<2$

\section{When 3 is ordinary}
In this section, we suppose that $N$ is a restrained odd perfect number such that $3^2\parallel N$ and we show that no such $N$ exists. 
Below is the factor tree derived from the component $3^2$, without the subtrees corresponding to the notable component.\\

$3^2\Longrightarrow 13$\\
\hspace*{8mm} $13^e\Longrightarrow2\cdot7$\\
\hspace*{12mm} $7^2\Longrightarrow3\cdot19$\\
\hspace*{16mm} $19^2\Longrightarrow3\cdot127$\\
\hspace*{20mm} $127^2\Longrightarrow3\cdot5419$\ \ \ XS 3\\
\hspace*{20mm} $127^q$\\
\hspace*{16mm} $19^q$\\
\hspace*{12mm} $7^q$\\
\hspace*{8mm} $13^2\Longrightarrow3\cdot61$\\
\hspace*{12mm} $61^e\Longrightarrow2\cdot31$\\
\hspace*{16mm} $31^2\Longrightarrow3\cdot331$\\
\hspace*{20mm} $331^2\Longrightarrow3\cdot7\cdot5233$\ \ \ XS 3\\
\hspace*{20mm} $331^q$\\
\hspace*{16mm} $31^q$\\
\hspace*{12mm} $61^2\Longrightarrow3\cdot13\cdot97$\\
\hspace*{16mm} $97^e\Longrightarrow2\cdot7^2$\\
\hspace*{20mm} $7^2\Longrightarrow3\cdot19$\ \ \ XS 3\\
\hspace*{20mm} $7^q$\\
\hspace*{16mm} $97^2\Longrightarrow3\cdot3169$\ \ \ XS 3\\
\hspace*{16mm} $97^q$\\
\hspace*{12mm} $61^q$\\
\hspace*{8mm} $13^q$\\

This tree shows that, in addition to $3^2\parallel N$, we have $13\mid N$ and the notable prime is in $S=\acc{7, 13, 19, 31, 61, 97, 127, 331}$. In particular, $n\equiv1\pmod3$.

% Every unfinished branch ends by specifying $n$.
% We have the following strategy to cut a branch when $p$ is also known.

% We want to rule out small primes as factors of $N$. To rule out primes in $S$,
% are difficult 
% We start by ruling out small primes as factor of $N$.

First, we consider the ordinary primes that produce exclusively prime numbers in $\acc{3}\cup S\setminus\acc{97}$. Using the program all.cpp, we find the set $R$ of solutions to 
$\rad\paren{\sigma\paren{x^2}}\mid3\cdot7\cdot13\cdot19\cdot31\cdot61\cdot127\cdot331$ where $x$ is an odd prime and $x<10^{10}$.

\begin{table}[!htb]
\centerline{
\begin{tabular}{|r|l|l|l|l|}
\hline
$x$ & $x\in S$? & $x\in R$? & $\sigma\paren{x^2}$ & Order of elimination\\
\hline
\hline
3 & & yes & 13 &  \\ \hline
5 & & yes & 31 & 1 \\ \hline
7 & yes & yes & $3\cdot19$ &  \\ \hline
11 & & yes & $7\cdot19$ & 2 \\ \hline
13 & yes & yes & $3\cdot61$ &  \\ \hline
19 & yes & yes & $3\cdot127$ & 8 \\ \hline
31 & yes & yes & $3\cdot331$ & 10 \\ \hline
67 & & yes & $3\cdot7^2\cdot31$ & 3 \\ \hline
107 & & yes & $7\cdot13\cdot127$ & 4 \\ \hline
127 & yes & & $3\cdot5419$ & 6 \\ \hline
191 & & yes & $7\cdot13^2\cdot31$ & 9 (ruled out in 3p2.txt)  \\ \hline
331 & yes & & $3\cdot7\cdot5233$ & 7 \\ \hline
653 & & yes & $7\cdot13^2\cdot19^2$ & 5 \\ \hline
2819 & & yes & $19^4\cdot61$ &  \\ \hline
\end{tabular}}
\caption{This useful table contains 127, 331, and every prime in $R$, as well as the order in which they are ruled out.}\label{tab1}
\end{table}

We rule out some of the primes contained in Table 2, in the given order. 

\subsection*{Ruling out 5}

\hspace*{8mm}$5^e$\ \ \ Lemma~\ref{p1mod3}\\
\hspace*{8mm}$5^2\Longrightarrow 31$\\
\hspace*{12mm}$3^2\Longrightarrow 13$\\
\hspace*{16mm}$13^e\Longrightarrow2\cdot7\ \ \ \sigma_{-1}\paren{3^2\cdot5^2\cdot7^2}>2$\\
\hspace*{16mm}$13^2\Longrightarrow3\cdot61\ \ \ \sigma_{-1}\paren{3^2\cdot5^2\cdot13^2\cdot31^2}>2$\\
\hspace*{16mm}$13^q\ \ \ \sigma_{-1}\paren{3^2\cdot5^2\cdot13^2\cdot31^2}>2$

\subsection*{Ruling out 11}

\hspace*{4mm} $11^2\Longrightarrow7\cdot19$\\
\hspace*{8mm} $3^2\Longrightarrow13\ \ \ \sigma_{-1}\paren{3^2\cdot7^2\cdot11^2\cdot13\cdot19^2}>2$

\subsection*{Ruling out 67}

\hspace*{4mm} $67^2\Longrightarrow3\cdot7^2\cdot31$\\
\hspace*{8mm} $7^2\Longrightarrow3\cdot19\ \ \ \sigma_{-1}\paren{3^2\cdot7^2\cdot13\cdot19^2\cdot31^2\cdot67^2}>2$\\
\hspace*{8mm} $7^q$\\
\hspace*{12mm} $31^2\Longrightarrow3\cdot331$\\
\hspace*{16mm} $331^2\Longrightarrow3\cdot7\cdot5233$\ \ \ XS 3

\subsection*{Ruling out 107}

\hspace*{4mm} $107^2\Longrightarrow7\cdot13\cdot127$\\
\hspace*{8mm} $127^2\Longrightarrow3\cdot5419$\\
\hspace*{12mm} $5419^2\Longrightarrow3\cdot31\cdot313\cdot1009$\\
\hspace*{16mm} $1009^1\Longrightarrow2\cdot5\cdot101$\ \ \ F 5\\
\hspace*{16mm} $1009^2\Longrightarrow3\cdot37\cdot9181$\ \ \ XS 3\\
\hspace*{8mm} $127^q$\\
\hspace*{12mm} $7^2\Longrightarrow3\cdot19$\\
\hspace*{16mm} $19^2\Longrightarrow3\cdot127$\\
\hspace*{20mm} $13^e$\ \ \ See below\\
\hspace*{20mm} $13^2\Longrightarrow3\cdot61$\ \ \ XS 3

We need to handle the branch leading to $13^e$, which contains the components $3^2$, $7^2$, $13^e$, $19^2$, and $127^q$. Since $107\equiv2\pmod3$, 107 cannot be produced by an ordinary component. 
Since $\operatorname{ord}_{107}(127)=106$, Lemma~\ref{NotFromNotable} ensures that $127^q$ does not produce 107. So $107$ must be produced by the special component $13^e$. Since $\operatorname{ord}_{107}(13)=\operatorname{ord}_{194723}(13)=53$,
Lemma~\ref{Robin} implies that $194723\mid N$ and we can finish the branch.\\
\hspace*{4mm} $107^2\Longrightarrow7\cdot13\cdot127$\\
\hspace*{8mm} $127^q$\\
\hspace*{12mm} $7^2\Longrightarrow3\cdot19$\\
\hspace*{16mm} $19^2\Longrightarrow3\cdot127$\\
\hspace*{20mm} $13^e\Longrightarrow2\cdot7\cdot107\cdot194723\cdot...$\\
\hspace*{24mm} $194723^2\Longrightarrow7\cdot13\cdot19\cdot21930157$\ \ \ XS 7

% \subsection*{Ruling out 191}

% \hspace*{4mm} $191^2\Longrightarrow7\cdot13^2\cdot31$\\
% \hspace*{8mm} $7^2\Longrightarrow3\cdot19$\\
% \hspace*{12mm} $19^2\Longrightarrow3\cdot127\ \ \ \sigma_{-1}\paren{3^2\cdot7^2\cdot13^2\cdot19^2\cdot31^2\cdot127^2\cdot191^2}>2$\\
% \hspace*{12mm} $19^q$\\
% \hspace*{16mm} $31^2\Longrightarrow3\cdot331\ \ \ \sigma_{-1}\paren{3^2\cdot7^2\cdot13^2\cdot19^2\cdot31^2\cdot191^2\cdot331^2}>2$\\
% \hspace*{8mm} $7^q$\\
% \hspace*{12mm} $31^2\Longrightarrow3\cdot331$\\
% \hspace*{16mm} $331^2\Longrightarrow3\cdot7\cdot5233$\\
% \hspace*{20mm} $5233^e\Longrightarrow2\cdot2617$\\
% \hspace*{24mm} $2617^2\Longrightarrow3\cdot193\cdot11833$\ \ \ XS 3\\
% \hspace*{20mm} $5233^2\Longrightarrow3\cdot7\cdot31\cdot42073$\ \ \ XS 3\\

\subsection*{Ruling out 653}

\hspace*{4mm} $653^e$\ \ \ Lemma~\ref{p1mod3}\\
\hspace*{4mm} $653^2\Longrightarrow7\cdot13^2\cdot19^2$\\
\hspace*{8mm} $7^2\Longrightarrow3\cdot19$\\
\hspace*{12mm} $19^2$\ \ \ XS 19\\
\hspace*{12mm} $19^q$\\
\hspace*{16mm} $3^2\Longrightarrow13$\\
\hspace*{20mm} $13^e\Longrightarrow2\cdot7$\ \ \ See below\\
\hspace*{20mm} $13^2$\ \ \ XS 13\\
\hspace*{8mm} $7^q$\\
\hspace*{12mm} $19^2\Longrightarrow3\cdot127$\\
\hspace*{16mm} $127^2\Longrightarrow3\cdot5419$\\
\hspace*{20mm} $5419^2\Longrightarrow3\cdot31\cdot313\cdot1009$\ \ \ XS 3

We need to handle the branch leading to $13^e$, which contains the components $3^2$, $7^2$, $13^e$, $19^q$, and $653^2$. Since $653\equiv2\pmod3$, $653$ cannot be produced by an ordinary component. 

Suppose for the sake of contradiction that $653$ is produced by the notable component $19^q$.
Consider the $51$-digit prime factor of $19^{163}-1$:
$$s = 342824539081437076343940054698469728805379685794103$$
Notice that $\operatorname{ord}_{653}(19)=\operatorname{ord}_s(19)=163$. By Lemma~\ref{Robin}, our assumption $653\mid\sigma(19^q)$ implies $s\mid\sigma(19^q)$. Thus, $s$ is a factor of $N$.
Then we can finish the branch for $653\mid\sigma(19^q)$.\\

$653^2\Longrightarrow7\cdot13^2\cdot19^2$\\
\hspace*{8mm} $7^2\Longrightarrow3\cdot19$\\
\hspace*{12mm} $19^q\Longrightarrow653\cdot s\cdot...$\\
\hspace*{16mm} $13^e\Longrightarrow2\cdot7$\\
\hspace*{20mm} $s^2\Longrightarrow64513588025437\cdot266028891483340333\cdot P_{70}$\\
\hspace*{24mm} $64513588025437^2\Longrightarrow3\cdot...$\\
\hspace*{28mm} $266028891483340333^2\Longrightarrow3\cdot...$\ \ \ XS 3\\
Therefore 653 must be produced by the special component.
Then Lemma~\ref{Robin} implies that $13^e$ also produces 7499.
Now we can finish the branch for $653\mid\sigma(13^e)$.

$653^2\Longrightarrow7\cdot13^2\cdot19^2$\\
\hspace*{8mm} $7^2\Longrightarrow3\cdot19$\\
\hspace*{12mm} $19^q$\\
\hspace*{16mm} $13^e\Longrightarrow2\cdot7\cdot653\cdot7499\cdot...$\\
\hspace*{20mm} $7499^2\Longrightarrow7\cdot1093\cdot7351$\ \ \ XS 7

\subsection*{Ruling out 127}
First, we rule out $127^2$.\\

$127^2\Longrightarrow3\cdot5419$\\
\hspace*{8mm} $5419^2\Longrightarrow3\cdot31\cdot313\cdot1009$\\
\hspace*{12mm} $1009^e\Longrightarrow2\cdot5\cdot101$\ \ \ F 5\\
\hspace*{12mm} $1009^2\Longrightarrow3\cdot37\cdot9181$\ \ \ XS 3\\
Now, we consider the branch leading to $127^q$.\\

$3^2\Longrightarrow 13$\\
\hspace*{8mm} $13^e\Longrightarrow2\cdot7$\\
\hspace*{12mm} $7^2\Longrightarrow 3\cdot19$\\
\hspace*{16mm} $19^2\Longrightarrow 3\cdot127$\\
\hspace*{20mm} $127^q$

Since the components $7^2$ and $19^2$ produce the two copies of 3
needed by the component $3^2$, no other ordinary component can produce a copy of 3.
That is, every other ordinary prime is congruent to $2\pmod3$. Thus the only
prime factors of $N$ that are congruent to $1\pmod{3}$ are 7, 13, 19, and 127. 
The results in Table 2 show that there remains no solution to
$\rad\paren{\sigma\paren{x^2}}\mid7\cdot13\cdot19\cdot127$ where $x$ is prime and $x<10^{10}$.
% \begin{itemize}
% \item $\sigma\paren{11^2}=7\cdot19$
% \item $\sigma\paren{107^2}=7\cdot13\cdot127$
% \item $\sigma\paren{653^2}=7\cdot13^2\cdot19^2$
% \end{itemize}
% However, 11, 107, and 653 are ruled out.
Thus the only ordinary primes smaller than $10^{10}$ are 3, 7, and 19.
Let us review the components of $N$ to bound the abundancy. $N$ only contains
\begin{itemize}
\item the components $3^2$, $7^2$, $13^e$, $19^2$, $127^q$,
\item at most 2 irregular components producing the second copy of 7 and of 19,
\item at most 14 regular components with $a+b\le6$ (see section~2.3),
\item and regular components with $a+b\ge7$ (see section~2.4).
\end{itemize}
The abundancy of $N$ is thus at most
$$\frac{13}{3^2}\cdot\frac{3\cdot19}{7^2}\cdot\frac{13}{12}\cdot\frac{3\cdot127}{19^2}\cdot\frac{127}{126}\cdot\paren{1+\frac1{10^{10}}}^{2+14}\cdot A_{\infty}$$
which is at most $1.9592144484846309$.
%ochem@ochem-laptop:~/opn/JZ1$ calc 13/9*57/49*13/12*3*127/19/19*127/126*1.0000000001^16*1.01178747704420359446
%        1.95921444848463087426

\subsection*{Ruling out 331}
First, we rule out $331^2$.\\

$331^2\Longrightarrow3\cdot7\cdot5233$\\
\hspace*{8mm} $5233^e\Longrightarrow2\cdot2617$\\
\hspace*{12mm} $2617^2\Longrightarrow3\cdot193\cdot11833$\\
\hspace*{16mm} $193^2\Longrightarrow3\cdot7\cdot1783$\ \ \ XS 3\\
\hspace*{8mm} $5233^2\Longrightarrow3\cdot7\cdot31\cdot42073$\\
\hspace*{12mm} $42073^e\Longrightarrow2\cdot109\cdot193$\\
\hspace*{16mm} $193^2\Longrightarrow3\cdot7\cdot1783$\ \ \ XS 3\\
\hspace*{12mm} $42073^2\Longrightarrow3\cdot19\cdot409\cdot75931$\ \ \ XS 3\\
Now, we consider the branch leading to $331^q$.\\

$3^{2}\Longrightarrow 13$\\
\hspace*{8mm} $13^2\Longrightarrow3\cdot61$\\
\hspace*{12mm} $61^e\Longrightarrow2\cdot31$\\
\hspace*{16mm} $31^2\Longrightarrow3\cdot331$\\
\hspace*{20mm} $331^q$

Since the components $13^2$ and $31^2$ produce the two copies of 3
needed by the component $3^2$, no other ordinary component can produce a copy of 3.
That is, every other ordinary prime is congruent to $2\pmod{3}$. Thus the only
prime factors of $N$ that are congruent to $1\pmod{3}$ are 13, 31, 61, and 331. 
The results in Table 2 show that there is no solution left to
$\rad\paren{\sigma\paren{x^2}}\mid13\cdot31\cdot61\cdot331$ where $x$ is prime and $x<10^{10}$.
Therefore the only ordinary primes smaller than $10^{10}$ are 3, 13, and 31.
Let us review the components of $N$ to bound the abundancy. $N$ only contains
\begin{itemize}
\item the components $3^2$, $13^2$, $31^2$, $61^e$, $331^q$, 
\item at most 2 irregular components producing the second needed copy of 13 and of 31,
\item at most 14 regular components with $a+b\le6$,
\item and regular components with $a+b\ge7$.
\end{itemize}
The abundancy of $N$ is thus at most
$$\frac{13}{3^2}\cdot\frac{3\cdot61}{13^2}\cdot\frac{3\cdot331}{31^2}\cdot\frac{61}{60}\cdot\frac{331}{330}\cdot\paren{1+\frac1{10^{10}}}^{2+14}\cdot A_{\infty}$$
which is at most 1.6675275884645719.
% ochem@ochem-laptop:~/opn/JZ1$ calc 13/9*183/169*993/31^2*61/60*331/330*1.0000000001^16*1.0117874770442035944654
%         1.66752758846457182241

\subsection*{Ruling out 19}
First, 19 is not ordinary since $\sigma\paren{19^2}=3\cdot127$ and 127 is ruled out.
Now, we consider the branch leading to $19^q$.\\

$3^2\Longrightarrow 13$\\
\hspace*{8mm} $13^e\Longrightarrow2\cdot7$\\
\hspace*{12mm} $7^2\Longrightarrow3\cdot19$\\
\hspace*{16mm} $19^q$

So the known components of $N$ are $3^2$, $7^2$, $13^e$, and $19^q$.
Since $\omega(N)\ge10$ by Theorem~\ref{Pace}, $N$ also contains at least
six other ordinary components.
Since the component $7^2$ produces the first copy of $3$, at most one other copy of 3
may be produced by an ordinary component. Therefore $N$ contains at most one prime $s>19$ congruent to $1\pmod3$. Thus the set of prime factors of $N$ congruent to $1\pmod3$ is $\acc{7,13,19}\cup\acc{s}$. So  $\operatorname{rad}\paren{\sigma\paren{s^2}}\mid3\cdot7\cdot13\cdot19$.
From the data in Table 2, we deduce that $s>10^{10}$.
Now, for every prime $r>19$ other than $s$, we have
$\operatorname{rad}\paren{\sigma\paren{r^2}}\mid7\cdot13\cdot19\cdot s$.
To get a lower bound on $r$, we consider two cases:
\begin{itemize}
\item If $s\nmid\sigma\paren{r^2}$, then $\operatorname{rad}\paren{\sigma\paren{r^2}}\mid7\cdot13\cdot19$. Thus $r>10^{10}$ by the previous argument.
\item If $s\mid\sigma\paren{r^2}$, then $\sigma\paren{r^2}\ge s$. Thus $r>\sqrt{s}>10^5$.
\end{itemize}
In any case, if $x>19$ is a prime factor of $N$, then $x>10^5$.
Let us review the components of $N$ to bound the abundancy. $N$ only contains
\begin{itemize}
\item the components $3^2$, $7^2$, $13^e$, $19^q$, 
\item at most 4 irregular components producing the second copy of 3, the second copy of 7, and the two copies of $s$,
\item at most 14 regular components with $a+b\le6$,
\item and regular components with $a+b\ge7$.
\end{itemize}
The abundancy of $N$ is thus at most
$$\frac{13}{3^2}\cdot\frac{3\cdot19}{7^2}\cdot\frac{13}{12}\cdot\frac{19}{18}\cdot\paren{1+\frac1{10^5}}^{4+14}\cdot A_{\infty}$$
which is at most 1.944420956068104.

\subsection*{Ruling out primes up to 1000}
The file 3p2.txt rules out every prime not in $\acc{3, 7, 13, 31, 61, 97, 193}$ up to 1000.

\subsection*{Ruling out 31}
First, 31 is not ordinary since $\sigma\paren{31^2}=3\cdot331$ and 331 is ruled out.
Now, we consider the branch leading to $31^q$.\\

$3^2\Longrightarrow 13$\\
\hspace*{8mm} $13^2\Longrightarrow3\cdot61$\\
\hspace*{12mm} $61^e\Longrightarrow2\cdot31$\\
\hspace*{16mm} $31^q$

So the known components of $N$ are $3^2$, $13^2$, $31^q$, and $61^e$.
Since $\omega(N)\ge10$ by Theorem~\ref{Pace}, $N$ also contains at least
six other ordinary components.
Since the component $13^2$ produces the first copy of $3$, at most one other copy of 3
may be produced by an ordinary component. Therefore $N$ contains at most one unknown prime $s\equiv1\pmod3$. Thus the set of prime factors of $N$ congruent to $1\pmod3$ is $\acc{13,31,61}\cup\acc{s}$.
From the previous results, only 7, 97, or 193 can be an unknown prime factor of $N$ smaller than 1000.
\begin{itemize}
\item 7 cannot be a factor since $\sigma\paren{7^2}=3\cdot19$ and 19 is ruled out.
\item 97 cannot be a factor since $\sigma\paren{97^2}=3\cdot3169$
and $\sigma\paren{3169^2}=3\cdot3348577$ would create three copies of 3.
\item 193 cannot be a factor since $\sigma\paren{193^2}=3\cdot7\cdot1783$ and $\sigma\paren{1783^2}=3\cdot829\cdot1279$ would create three copies of 3.
\end{itemize}
% So  $\operatorname{rad}\paren{\sigma\paren{s^2}}\mid3\cdot13\cdot31\cdot61$.
% From the data in Table 2, we deduce that $s>10^{10}$.
% Now, for every prime $r>19$ other than $s$, we have
% $\operatorname{rad}\paren{\sigma\paren{r^2}}\mid7\cdot13\cdot19\cdot s$.
% To get a lower bound on $r$, we consider two cases:
% \begin{itemize}
% \item If $s\nmid\sigma\paren{r^2}$, then $\operatorname{rad}\paren{\sigma\paren{r^2}}\mid7\cdot13\cdot19$. Thus $r>10^{10}$ by the previous argument.
% \item If $s\mid\sigma\paren{r^2}$, then $\sigma\paren{r^2}\ge s$. Thus $r>\sqrt{s}>10^5$.
% \end{itemize}
% In any case, if $x>19$ is a prime factor of $N$, then $x>10^5$.
Let us review the components of $N$ to bound the abundancy. $N$ only contains
\begin{itemize}
\item the components $3^2$, $13^2$, $31^q$, $61^e$, 
\item at most 4 irregular components producing the second copy of 3, the second copy of 13, and the two copies of $s$,
\item at most 14 regular components with $a+b\le6$,
\item and regular components with $a+b\ge7$.
\end{itemize}
The abundancy of $N$ is thus at most
$$\frac{13}{3^2}\cdot\frac{3\cdot61}{13^2}\cdot\frac{31}{30}\cdot\frac{61}{60}\cdot\paren{1+\frac1{10^3}}^{4+14}\cdot A_{\infty}$$
which is at most 1.69272709609736065.
%ochem@ochem-laptop:~/opn/JZ1$ calc 13/9*3*61/169*31/30*61/60*1.001^18*1.0117874770442035944654
%        1.69272709609736064223

\subsection*{Final argument}
At this point, every prime factor $f$ of $N$ is such that $f\in\acc{3,7,13,61,97,193}$ or $f>1000$.
At most two ordinary primes $x_1$ and $x_2$ produce the two copies of 3, so that $x_1$ and $x_2$ are the only ordinary primes congruent to $1\pmod3$.
For every ordinary prime $x$, $\rad\paren{\sigma\paren{x^2}}\mid3\cdot x_1\cdot x_2\cdot p\cdot n$. So $N$ contains at most 8 irregular primes, namely $p$, $n$,
and at most 6 ordinary primes that produce the two copies of 3, $x_1$, and $x_2$.
Every other prime is regular. We know that 3 is ordinary but we do not know the status of 7, 13, 61, 97, and 193. However, we can upper bound the abundancy of $N$ by
$$\frac{13}{3^2}\cdot\frac76\cdot\frac{13}{12}\cdot\frac{61}{60}\cdot\frac{97}{96}\cdot\frac{193}{192}\cdot\paren{1+\frac1{10^3}}^{8+14}\cdot A_{\infty}$$
which is at most 1.9497723588476532.
% ochem@ochem-laptop:~/opn/JZ1$ calc 13/9*7/6*13/12*61/60*97/96*193/192*1.001^22*1.0117874770442035944654
%         1.94977235884765319788

\section{When 3 is notable}
If $N$ is a restrained odd perfect number, then we have shown that $3\mid N$
in Section 2 and that $3^4\mid N$ in Section 3.
Therefore, if $N$ exists, then 3 must be the notable prime.
In this section, we show that the notable component $3^q$ satisfies $q\ge23,000,000,000$. 
First, we need this relation between $q$ and $\omega(N)$.
% \begin{lemma}\label{omegan}
% If $n=3$, then $q\ge\tfrac{\omega(N)}3-1$.
% \end{lemma}

% \begin{proof}
% Let $O_1$ and $O_2$ denote the sets of ordinary primes that are congruent to $1\pmod{3}$ and to $2\pmod{3}$, respectively.
% We set $o_1=|O_1|$ and $o_2=|O_2|$. By Lemma~\ref{rtt}, for every prime $s\in O_2$, either $\sigma(s^2)=p$ or $r \mid \sigma(s^2)$ with $r\in O_1$. By considering the number of copies of primes in $O_1$ produced by primes in $O_2$, we get $2o_1\ge o_2-1$.
% So $\omega(N)=2+o_1+o_2\le2+o_1+(2o_1+1)=3(o_1+1)$.
% Since $3 \mid \sigma(r^2)$ when $r\in O_1$,
% we obtain $q\ge o_1\ge\tfrac{\omega(N)}3-1$.
% \end{proof}

\begin{lemma} If $n=3$, then $q\ge\tfrac{\omega(N)}2-2$. \label{omegan}
\end{lemma}
\begin{proof}
Let $O_1$ and $O_2$ denote the sets of ordinary primes that are congruent to $1\pmod{3}$ and to $2\pmod{3}$, respectively.
We set $o_1=|O_1|$ and $o_2=|O_2|$. By Lemma~\ref{rtt}, for every prime $s\in O_2$, either $\sigma(s^2)=p$ or $r\mid\sigma(s^2)$ with $r\in O_1$.
By Theorem 2 in~\cite{BugeaudMihailescu:2007}, the equation $x^2+x+1=3y^t$ has no integer solution for $t\ge3$, and thus for any prime $s\in O_1$, except at most two of them (one where $\sigma(s^2)=3p$ and one where $\sigma(s^2)=3p^2$), we have $r\mid\sigma(s^2)$.
Therefore, for every ordinary prime $s$, we have $r\mid\sigma(s^2)$ except maybe in the three cases where $\sigma(s^2)\in\acc{p,3p,3p^2}$.
Now, by Lemma 3 in~\cite{OchemRaoOmega}, we cannot have the cases $p$ and $3p$ simultaneously.
So we have $r\mid\sigma(s^2)$ except in at most two exceptional cases.
This gives $2o_1\ge o_2+o_1-2$, which means that $o_2\le o_1+2$.
So $\omega(N)=2+o_1+o_2\le2+o_1+(o_1+2)=2(o_1+2)$.
Since $3\mid\sigma(r^2)$ when $r\in O_1$, we obtain $q\ge o_1\ge\tfrac{\omega(N)}2-2$.
\end{proof}

Now we rule out 5 and 7.\\

$5^e$\ \ \ Lemma~\ref{p1mod3}\\
\hspace*{4mm} $5^2\Longrightarrow31$\\
\hspace*{8mm} $31^2\Longrightarrow3\cdot331$\\
\hspace*{12mm} $331^2\Longrightarrow3\cdot7\cdot5233\ \ \ \sigma_{-1}\paren{3^2\cdot5^2\cdot7^2}>2$\\

$7^2\Longrightarrow3\cdot19$\\
\hspace*{8mm} $19^2\Longrightarrow3\cdot127$\\
\hspace*{12mm} $127^2\Longrightarrow3\cdot5419$\\
\hspace*{16mm} $5419^2\Longrightarrow3\cdot31\cdot313\cdot1009$\\
\hspace*{20mm} $313^e\Longrightarrow2\cdot157$\\
\hspace*{24mm} $157^2\Longrightarrow3\cdot8269$\\
\hspace*{28mm} $8269^2\Longrightarrow3\cdot7\cdot3256411$\\
\hspace*{32mm} $3256411^2\Longrightarrow3\cdot7^2\cdot733\cdot1543\cdot63781$\ \ \ XS 7\\
\hspace*{20mm} $313^2\Longrightarrow3\cdot181^2$\\
\hspace*{24mm} $181^e\Longrightarrow2\cdot7\cdot13\ \ \ \sigma_{-1}\paren{3^4\cdot7^2\cdot13^2\cdot19^2\cdot31^2}>2$\\
\hspace*{24mm} $181^2\Longrightarrow3\cdot79\cdot139$\\
\hspace*{28mm} $79^2\Longrightarrow3\cdot7^2\cdot43$\\
\hspace*{32mm} $\sigma_{-1}\paren{3^6\cdot7^2\cdot19^2\cdot31^2\cdot43^2\cdot79^2\cdot127^2\cdot139^2}>2$\\

% 7^2 => 3 19
%  19^2 => 3 127
%   127^2 => 3 5419
%    5419^2 => 3 31 313 1009
%     313^e => 2 157
%      157^2 => 8269
%       8269^2 => 3 7 3256411
%        3256411^2 => 3 7^2 733 1543 63781    XS 7
%     313^2 => 3 181^2
%      181^e => 2 7 13    a(3^4*7^2*13^2*19^2*31^2) = 2.05204625901750125902117182648
%      181^2 => 3 79 139
%       79^2 => 3 7^2 43    a(3^6*7^2*19^2*31^2*43^2*79^2*127^2*139^2) = 2.00229872301619017614388428651

Then in the file tree.txt, we rule out every prime from 11 up to $3,000,000$.
This file has been generated by the program tree.cpp using a few helpful factorizations in checkfacts.txt.
The program does not test for the contradictions XS and $\sigma_{-1}>2$ since they are not necessary for this computation.

Finally, we use the C++/GMP program loopless.cpp running the algorithm below to get a lower bound on $\omega(N)$. Let $F(k)$ denote the set of prime factors of $k$ distinct from 2 and 3. We maintain a set $M_1$ of primes that cannot be ordinary and a set $M_2$ of primes congruent to $1\pmod3$ that cannot divide $N$.
\begin{algorithm}
\begin{algorithmic}[1]
\caption{Sieving ordinary primes for $n=3$}
\State Initialize $M_1 \leftarrow \{ r\text{ is prime and } 5 \le r < 3,000,000 \}$
\State Initialize $M_2 \leftarrow \{ r\text{ is prime, }  7 \le r < 3,000,000 \text{, and } r \equiv 1 \pmod 3 \}$
\State Initialize $A \leftarrow 1.0$ \Comment{Abundancy tracker for surviving ordinary primes}
\State Initialize $k \leftarrow 0$ \Comment{Counting surviving ordinary primes}
\For{each prime $r > 3,000,000$ in increasing order}
    \If{$F(\sigma(r^2))\cap M_2 == \emptyset$} \Comment{$r$ can be ordinary}
        \State $A\leftarrow A\cdot\frac{r}{r-1}$
        \State $k\leftarrow k+1$
    \Else
        \State $M_1\leftarrow M_1\cup\{r\}$ \Comment{$r$ is not ordinary, so $r\nmid N$ or $r$ is special}
        \If{$r\equiv1\pmod3\text{ and }(r\equiv3\pmod4\text{ or }F(\sigma(r))\cap M_1\ne\emptyset)$}
                \State $M_2\leftarrow M_2\cup \{r\}$ \Comment{$r$ is not special either, so $r\nmid N$}
        \EndIf
    \EndIf
\EndFor
\end{algorithmic}
\end{algorithm}
We let the program run for 3 weeks until it said to need more than $2,500,000,000$ ordinary primes
up to $z=216,171,986,291$ only to reach abundancy $1.088507187872$.
Suppose that $N$ contains only $43,500,000,000$ prime factors greater than $z$.
Their abundancy is at most $(1+z^{-1})^{43,500,000,000}$.
Since primes smaller than $3,000,000$ are ruled out, we get $p>6,000,000$,
so that the abundancy of the special component is $\sigma_{-1}(p^e)<1+\tfrac1{6,000,000}$.
Moreover, the abundancy of the notable component is $\sigma_{-1}(3^q)<\tfrac32$.
The product of these abundancies does not reach 2:
$$\tfrac32\cdot\paren{1+\tfrac1{6,000,000}}\cdot1.076855990994\cdot(1+z^{-1})^{43,500,000,000}\approx1.99209763359616948564$$
Therefore, $N$ contains more than $2,500,000,000+43,500,000,000=46,000,000,000$ ordinary primes and Lemma~\ref{omegan} implies that $q\ge23,000,000,000$.

\section{Concluding remarks}
A better understanding of numbers of the form $\sigma(x^2)$ where $x$ is prime would make life easier.

\begin{question} Does there exist a prime $x$ such that $x^2+x+1$ is square-full?\label{sqfull}
\end{question}
Recall that $n$ is square-full if and only if $\paren{\rad(n)}^2\mid n$.
A negative answer to Question~\ref{sqfull} would reduce the possible pairs $(a,b)$
to $(0,1)$, $(1,0)$, $(1,1)$, $(1,b)$ with $b\ge2$, and $(a,1)$ with $a\ge2$.

\begin{question} Do there exist primes $x$, $y$, $p$ and $r$ and integers $a>b>0$ such that $x^2+x+1 = p^a r$ and $y^2+y+1=p^b r$?\label{thing that would make life easier}
\end{question}

If the answer to~\ref{thing that would make life easier} is ``no,'' then we could substantially tighten Lemma \ref{omegan}. The difficulty in part stems from there being multiple almost examples. For example, $x = 30$, $y = 11$, $p=7$ and $r=19$ yields a solution with the exception that $30$ is obviously not prime. Similarly, we may take $x=462$, $y=9353$, $p= 409$ and $r=523$. Note that in this case, both $x$ and $y$ are composite, since  $9353=(47)(199)$.

This paper shows that the existence of a restrained OPN can be reduced to the case
when $3$ is notable.
Our bound on $q$ can certainly gain orders of magnitude using better arguments from sieve theory.
However, it seems that new ideas are required to completely rule out that $3$ is notable, that is, to show that every OPN has at least two even exponents greater than 2.

\end{document}